\documentclass[11pt,twoside]{amsart}
\usepackage{latexsym,amssymb,amsmath}

\numberwithin{equation}{section}
\def\demo{\noindent{\it Proof. }}
\newtheorem{theorem}{Theorem}[section]
\newtheorem{lemma}[theorem]{Lemma}

\newtheorem{corollary}[theorem]{Corollary}

\theoremstyle{definition}
\newtheorem{definition}[theorem]{Definition}

\newtheorem{example}[theorem]{Example}

\begin{document}


\title[On the generalized Hamming weights]
{On the generalized Hamming weights of certain Reed--Muller-type codes} 

\thanks{The first author was supported by COFAA-IPN and SNI,
Mexico. The second author was
supported by a scholarship from CONACyT, Mexico. 
The third author was supported by SNI, Mexico. 
}

\author[M. Gonz\'alez-Sarabia]{Manuel Gonz\'alez-Sarabia} 
\address{M.G. Sarabia. Instituto Polit\'ecnico Nacional, 
UPIITA, Av. IPN No. 2580,
Col. La Laguna Ticom\'an,
Gustavo A. Madero C.P. 07340,
 Ciudad de M\'exico. 
Departamento de Ciencias B\'asicas}
\email{mgonzalezsa@ipn.mx}

\author[D. Jaramillo]{Delio Jaramillo}
\address{
Departamento de
Matem\'aticas\\
Centro de Investigaci\'on y de Estudios
Avanzados del
IPN\\
Apartado Postal
14--740 \\
07000 Mexico City, D.F.
}
\email{delio.jaramillo@cimat.mx}

\author[R. H. Villarreal]{Rafael H. Villarreal}
\address{
Departamento de
Matem\'aticas\\
Centro de Investigaci\'on y de Estudios
Avanzados del
IPN\\
Apartado Postal
14--740 \\
07000 Mexico City, D.F.
}
\email{vila@math.cinvestav.mx}


\begin{abstract} 
There is a nice combinatorial formula of P. Beelen and M. Datta for the $r$-th generalized
Hamming weight of an affine cartesian code. 
Using this combinatorial formula we give an
easy to evaluate formula to compute the $r$-th generalized Hamming weight for
a family of affine cartesian codes. If $\mathbb{X}$ is a set of
projective points over a finite field we determine the basic
parameters and the generalized Hamming weights of the Veronese type
codes on $\mathbb{X}$ and their dual codes in terms of the
basic parameters and the generalized Hamming weights of the corresponding projective Reed--Muller-type
codes on $\mathbb{X}$ and their dual codes. 
\end{abstract}

\maketitle 

\section{Introduction}\label{intro-section}
Let $K=\mathbb{F}_q$ be a finite field and let $C$ be an $[m,\kappa]$-{\it linear
code} of {\it length} $m$ and {\it dimension} $\kappa$, 
that is, $C$ is a linear subspace of $K^m$ with $\kappa=\dim_K(C)$. The
multiplicative group of $K$ is denoted by $K^*$. 
The \textit{dual code} of $C$ is given by
$$
C^\perp:=\{b \in K^m\colon \langle b,c\rangle =0 \, 
\, \forall \, \,  c \in C
\},
$$
where $b=(b_1,\ldots,b_m)$, $c=(c_1,\ldots,c_m)$, and 
$\langle b,c\rangle= \sum_{i=1}^mb_i c_i$ is the inner product of $a$
and $b$. 

Fix an integer $1\leq r\leq \kappa$. 
Given a subcode $D$ of $C$ (that is, $D$ is a linear subspace of $C$),
the {\it support\/} $\chi(D)$ of $D$ is the set of non-zero positions of $D$, that is,  
$$
\chi(D):=\{i\,\vert\, \exists\, (a_1,\ldots,a_m)\in D,\, a_i\neq 0\}.
$$

The $r$-th {\it generalized Hamming weight\/} of $C$, denoted
$\delta_r(C)$, is the size of the smallest support of an
$r$-dimensional subcode \cite{helleseth,klove,wei}. Generalized Hamming weights
have been extensively studied; see
\cite{carvalho,ghorpade,GHW2014,Pellikaan,Johnsen,olaya,schaathun-willems,tsfasman,wei-yang,Yang} 
and the
references therein. The study of these weights is related to 
trellis coding, $t$--resilient functions, and was motivated by 
some applications from cryptography \cite{wei}. If $r=1$,
$\delta_1(C)$ is the \textit{minimum distance} of $C$ and is denoted
$\delta(C)$.

In this note we give explicit formulas for the generalized Hamming
weights of certain projective Reed-Muller-type codes and study the
basic parameters (length, dimension, minimum distance) and the
generalized Hamming weights of Veronese
type codes and their dual codes. 

These linear codes are constructed 
as follows. Let $\mathbb{P}^{s-1}$ be a projective space over $K$, let
$\mathbb{X}=\{[P_1],\ldots,[P_m]\}$ be a subset of $\mathbb{P}^{s-1}$ where
$m=|\mathbb{X}|$ is the cardinality of the set $\mathbb{X}$, $P_i\in
K^s$ for all $i$, and let
$S=K[t_1,\ldots,t_s]=\oplus_{d=0}^\infty S_d$ be a polynomial ring with
the standard grading, where $S_d$ is the $K$-vector space generated by the
homogeneous polynomials in $S$ of degree $d$. Fix a degree $d\geq 1$. For each $i$ there is
$h_i \in S_d$ such that $h_i(P_i) \neq 0$. Indeed suppose
$P_i=(a_1,\ldots,a_s)$, there is at least one $k \in \{1,\ldots,s\}$
such that $a_k \neq 0$. 
Setting $h_i=t_k^d$ one has that $h_i \in S_d$ and $h_i(P_i) \neq 0$.
Consider the evaluation map       
\begin{equation*}
{\rm{ev}}_d: S_d \longrightarrow K^m,\quad h \mapsto \left(
\frac{h(P_1)}{h_1(P_1)},\ldots,\frac{h(P_m)}{h_m(P_m)}\right).
\end{equation*}
\quad This is a linear map between the $K$-vector spaces $S_d$ and $K^m$.
The \textit{Reed--Muller--type-code} of order $d$ 
associated to $\mathbb{X}$ \cite{duursma-renteria-tapia,GRT}, 
denoted $C_\mathbb{X}(d)$, is the image of ${\rm{ev}}_d$, that is
$$
C_\mathbb{X}(d)=\left\{\left(\frac{h(P_1)}{h_1(P_1)},\ldots,
\frac{h(P_m)}{h_m(P_m)}\right): h \in S_d\right\}.
$$
\quad The $r$-th generalized hamming weight $\delta_r(C_\mathbb{X}(d))$ of
$C_\mathbb{X}(d)$ is sometimes denoted by 
$\delta_\mathbb{X}(d,r)$. If $r=1$, $\delta_\mathbb{X}(d,r)$ is the
minimum distance of $C_\mathbb{X}(d)$ and is denoted by
$\delta_\mathbb{X}(d)$. The map ${\rm ev}_d$ is independent of
the set of representatives $P_1,\ldots,P_m$ that we choose for the points of 
$\mathbb{X}$, and the basic parameters of $C_\mathbb{X}(d)$ are
independent of $h_1,\ldots,h_m$ \cite[Lemma~2.13]{hilbert-min-dis} and
so are the generalized Hamming weights of $C_\mathbb{X}(d)$
\cite[Remark~1]{dual-reed-muller}.

The basic parameters of $C_\mathbb{X}(d)$ are related to the algebraic
invariants of the quotient ring $S/I(\mathbb{X})$, where
$I(\mathbb{X})$ is the vanishing
ideal of $\mathbb{X}$ (see for example \cite{sarabia4,vaz,algcodes}). Indeed, the
dimension of $C_\mathbb{X}(d)$ is given by the Hilbert function
$H_\mathbb{X}$ of $S/I(\mathbb{X})$,
that is, 
$$H_\mathbb{X}(d):=\dim_K(S_d/I(\mathbb{X})_d)=\dim_K (C_\mathbb{X}(d)),
$$
the length $m=|\mathbb{X}|$ of $C_\mathbb{X}(d)$ is
the degree or the multiplicity of $S/I(\mathbb{X})$. Moreover, the
regularity index of $H_\mathbb{X}$ is the regularity of
$S/I(\mathbb{X})$ \cite[pp.~226, 346]{monalg-rev} and is denoted ${\rm reg}(S/I(\mathbb{X}))$. By the
Singleton bound \cite{tsfasman} one has $\delta_\mathbb{X}(d)=1$ for $d\geq {\rm
reg}(S/I(\mathbb{X}))$.   
Recall that the  $a$-invariant of $S/I(\mathbb{X})$, denoted $a_\mathbb{X}$, is 
the regularity index minus $1$. 

Let $A_1,\ldots,A_{s-1}$ be subsets of $K=\mathbb{F}_q$
and let  
$
\mathbb{X}:=[A_1 \times \cdots\times A_{s-1} \times \{1\}] \subset
\mathbb{P}^{s-1}
$
be a projective cartesian set, 
where $d_i=|A_i|$ for all $i=1,\ldots,s-1$ and $2 \leq d_1 \leq \cdots
\leq d_{s-1}$. The Reed--Muller-type code
$C_{\mathbb{X}}(d)$ is called an {\it affine cartesian code\/}
\cite{cartesian-codes}.

There is a recent expression for the $r$-th generalized Hamming weight
of an affine cartesian code \cite[Theorem~5.4]{GHWCartesian}, which
depends on the $r$-th monomial in ascending lexicographic order of a
certain family of monomials (see \cite{GHWCartesian} and the proof of
Theorem~\ref{sarabia-theorem}). Using
this result in Section~\ref{ghw-section} we give an easy to evaluate formula to compute the
$r$-th generalized Hamming weight for a family of affine cartesian 
codes (Theorem~\ref{sarabia-theorem}). Other formulas for the second generalized Hamming
weight of an affine cartesian code are given in
\cite[Theorems~9.3 and 9.5]{rth-footprint}. 

Let $k\geq 1$ be an integer and let $M_1,\ldots,M_N$ be the set of all
monomials in $S$ of degree $k$, where $N=\binom{k+s-1}{s-1}$.
The map 
$$
\rho_k\colon\mathbb{P}^{s-1}\rightarrow\mathbb{P}^{N-1},\ \ \ 
[x]\mapsto [(M_1(x),\ldots,M_N(x))],
$$
is called the $k$-th {\it Veronese embedding}. Given
$\mathbb{X}\subset\mathbb{P}^{s-1}$, the $k$-th 
{\it Veronese type code\/} of degree $d$ is
$C_{\rho_k(\mathbb{X})}(d)$, the Reed--Muller-type code of degree $d$
on $\rho_k(\mathbb{X})$.

In Section~\ref{Veronese-codes-section} we are able to show that the Reed--Muller-type code
$C_{\mathbb{X}}(kd)$ over the set $\mathbb{X}$ has the same basic
parameters and the same generalized Hamming weights as 
the Veronese type code $C_{\rho_k(\mathbb{X})}(d)$ over the set $\mathbb{X}$
for $k\geq 1$ and $d\geq 1$ (Theorem~\ref{veronese-codes}). 
As a consequence making
$\mathbb{X}=\mathbb{P}^{s-1}$ we recover 
a result of Renter\'\i a and Tapia-Recillas \cite[Proposition~1]{Renteria-Tapia-Veronese}. Also we show that the dual codes of $C_{\mathbb{X}}(kd)$ and  $C_{\rho_k(\mathbb{X})}(d)$ are equivalent (Theorem \ref{dual}).

For all unexplained
terminology and additional information  we refer to 
\cite{CLO,monalg-rev} (for the theory of Gr\"obner bases), and
\cite{MacWilliams-Sloane,tsfasman} (for the theory of
error-correcting codes and linear codes). 

\section{Generalized Hamming weights of some affine cartesian
codes}\label{ghw-section}

In this section we present our main result on Hamming
weights of certain cartesian codes. 
To avoid repetitions, we continue to employ
the notations and 
definitions used in Section~\ref{intro-section}.

Let  $\prec$ be a monomial
order on $S$ and let $(0)\neq I\subset S$ be an ideal. If $f$ is a non-zero 
polynomial in $S$, the {\it leading
monomial\/} of $f$ 
is denoted by ${\rm in}_\prec(f)$. The {\it initial ideal\/} of $I$, denoted by
${\rm in}_\prec(I)$,  is the monomial ideal given by 
$${\rm in}_\prec(I)=(\{{\rm in}_\prec(f)|\, f\in I\}).
$$ 

A monomial $t^a$ is called a 
{\it standard monomial\/} of $S/I$, with respect 
to $\prec$, if $t^a$ is not in the ideal ${\rm in}_\prec(I)$. 
The set of standard monomials, denoted $\Delta_\prec(I)$, is called the {\it
footprint\/} of $S/I$. The footprint of $S/I$ is also called the {\it
Gr\"obner \'escalier\/} of $I$. The image of the standard polynomials of
degree $d$, under the canonical map $S\mapsto S/I$, 
$x\mapsto \overline{x}$, is equal to $S_d/I_d$, and the 
image of $\Delta_\prec(I)$ is a basis of $S/I$ as a $K$-vector space. 
This is a classical result of Macaulay \cite[Chapter~5]{CLO}.  

We come to our main result. 

\begin{theorem}\label{sarabia-theorem}
Let $\mathbb{X}:=[A_1 \times \cdots\times A_{s-1} \times \{1\}]$ be a
subset of $\mathbb{P}^{s-1}$, where $A_i\subset\mathbb{F}_q$ and
$d_i=|A_i|$ for $i=1,\ldots,s-1$. If $2 \leq d_1 \leq \cdots \leq
d_{s-1}$ and $d\geq 1$, then
\begin{equation*}
\delta_r(C_\mathbb{X}(d))=\left\{
\begin{array}{lll}
d_{k+r+1} \cdots d_{s-1}[(d_{k+1}-\ell+1)d_{k+2} \cdots d_{k+r}-1] &
{\mbox{if}} &1\leq r<s-k-1, \\
& & \\
(d_{k+1}-\ell+1) d_{k+2} \cdots d_{s-1}-1 & {\mbox{if}} & 1\leq r=s-k-1,
\end{array} \right.
\end{equation*}
where we set $d_i \cdots 
d_j=1$ if $i>j$ or $i<1$, and $k\geq 0$, $\ell$ are
the unique integers such that $d=\sum_{i=1}^{k} (d_i-1)+\ell$ and $1 \leq \ell \leq d_{k+1}-1$. 
\end{theorem}

\demo Setting $n=s-1$, $R=K[t_1,\ldots,t_n]$ a polynomial ring with
coefficients in $K=\mathbb{F}_q$, and $L=(t_1^{d_1},\ldots,t_n^{d_{n}})$, we order the set
$M_{\leq d}:=\Delta_\prec(L)\cap R_{\leq d}$ of all standard
monomials of $R/L$ of degree at most $d$ with the lexicographic 
order (lex order for short), that is, $t^a\succ t^b$ if
and only if the first non-zero entry of $a-b$ is positive. For
$r>1$, $0 \leq k\leq n-r$, the $r$-th 
monomial $t_1^{b_{r,1}}\cdots t_n^{b_{r,n}}$ of $M_{\leq d}$ in
decreasing lex order is
$$
t_1^{d_1-1}\cdots t_k^{d_k-1}t_{k+1}^{\ell-1}t_{k+r}
$$
and the $r$-th  monomial $t_1^{a_{r,1}}\cdots t_n^{a_{r,n}}$ of
$M_{\geq c_0-d}:=\Delta_\prec(L)\cap R_{\geq c_0-d}$ in ascending lex order, where
$c_0=\sum_{i=1}^n(d_i-1)$, is 
$$
t_{k+1}^{d_{k+1}-\ell}t_{k+2}^{d_{k+2}-1}\cdots
t_{k+r-1}^{d_{k+r-1}-1} t_{k+r}^{d_{k+r}-2} t_{k+r+1}^{d_{k+r+1}-1}
\cdots t_n^{d_n-1}.
$$  

Case (I): $0 \leq k < n-r$. The case $r=1$ was proved in 
\cite[Theorem 3.8]{cartesian-codes}. Thus we may also assume $r\geq 2$. 
Therefore, applying \cite[Theorem 5.4]{GHWCartesian}, we obtain
that $\delta_r(C_\mathbb{X}(d))$ is given by 
\begin{align*}
&1+\sum_{i=1}^na_{r,i}\prod_{j=i+1}^nd_j=1+(d_{k+1}-\ell)d_{k+2}
\cdots d_n+\sum_{i=k+2, i \neq k+r}^n 
(d_{i}-1) \prod_{j=i+1}^n d_j\\
&+(d_{k+r}-2)d_{k+r+1} \cdots d_n \\ & = (d_{k+1}-\ell)d_{k+2}
\cdots d_n+\left(1+\sum_{i=k+2}^n 
(d_{i}-1) \prod_{j=i+1}^n d_j\right)-d_{k+r+1} \cdots d_n\\ 
&=(d_{k+1}-\ell)d_{k+2}
\cdots d_n+\left(d_{k+2}\cdots d_n\right)-d_{k+r+1} \cdots d_n\\
&=(d_{k+1}-\ell+1)d_{k+2}
\cdots d_n-d_{k+r+1} \cdots d_n=d_{k+r+1} \cdots d_n
[(d_{k+1}-\ell+1)d_{k+2} \cdots d_{k+r}-1].
\end{align*}

Case (II): $k=n-r$. In this case the $r$-th monomial $t_1^{a_{r,1}}\cdots t_n^{a_{r,n}}$ of
$M_{\geq c_0-d}$ in ascending lex order is 
$$
t_{k+1}^{d_{k+1}-\ell}t_{k+2}^{d_{k+2}-1}\cdots
t_{k+r-1}^{d_{k+r-1}-1} t_{k+r}^{d_{k+r}-2} t_{k+r+1}^{d_{k+r+1}-1}
\cdots t_n^{d_n-1}.
$$  

Therefore, applying \cite[Theorem 5.4]{GHWCartesian}, we obtain
that $\delta_r(C_\mathbb{X}(d))$ is given by 
\begin{align*}
&1+\sum_{i=1}^na_{r,i}\prod_{j=i+1}^nd_j=1+(d_{k+1}-\ell)d_{k+2}
\cdots d_n+\sum_{i=k+2}^{n-1} 
(d_{i}-1) \prod_{j=i+1}^n d_j+(d_n-2)\\
&=(d_{k+1}-\ell)d_{k+2}
\cdots d_n+\left(1+\sum_{i=k+2}^{n} 
(d_{i}-1) \prod_{j=i+1}^n d_j\right)-1\\
& = (d_{k+1}-\ell)d_{k+2}
\cdots d_n+\left(d_{k+2}\cdots d_n\right)-1=(d_{k+1}-\ell+1)d_{k+2}
\cdots d_n-1.\qed
\end{align*}

\begin{definition}\label{projectivetorus-def}\rm The set 
$\mathbb{T}=\{[(x_1,\ldots,x_s)]\in\mathbb{P}^{s-1}\vert\, x_i\in
K^*\, \forall\, i\}$ is called a {\it projective
torus\/}.
\end{definition}

\begin{corollary} \label{GHWtorus}
Let $\mathbb{T}$ be a projective torus in
$\mathbb{P}^{s-1}$ and let $\delta_r(C_\mathbb{T}(d))$ be the $r$-th
generalized Hamming weight of $C_\mathbb{T}(d)$. Then
$$
\delta_r(C_\mathbb{T}(d))=\left[(q-1)^{r-1}(q-\ell)-1\right](q-1)^{s-k-r-1}
$$
for $1\leq r\leq s-k-1$, 
where $d=k(q-2)+\ell$, $k\geq 0$,  
$1\leq \ell
\leq q-2$.
\end{corollary}

\begin{proof} It follows readily from Theorem~\ref{sarabia-theorem}
making $A_i=K^*=\mathbb{F}_q\setminus\{0\}$ for $i=1,\ldots,s-1$.  
\end{proof}

This corollary generalizes the case when $\mathbb{X}$ is a projective
torus in $\mathbb{P}^{s-1}$ and $r=1$: 

\begin{theorem}{\rm\cite[Theorem~3.5]{ci-codes}}\label{maria-vila-hiram-eliseo} 
Let $\mathbb{T}$ be a projective torus in
$\mathbb{P}^{s-1}$ and let $C_\mathbb{T}(d)$ be the Reed--Muller-type
code on $\mathbb{T}$ of degree $d\geq 1$. Then its length is
$(q-1)^{s-1}$, its minimum distance is given by 
$$
\delta_\mathbb{T}(d)=\left\{\begin{array}{cll}
(q-1)^{s-(k+2)}(q-1-\ell)&\mbox{if}&d\leq (q-2)(s-1)-1,\\
1&\mbox{if}&d\geq (q-2)(s-1),
\end{array}
 \right.
$$
where $k$ and $\ell$ are the unique integers such that $k\geq 0$,
$1\leq \ell\leq q-2$ and $d=k(q-2)+\ell$, and the regularity of
$S/I(\mathbb{T})$ is $(q-2)(s-1)$. 
\end{theorem}

The case when $\mathbb{X}$ is a 
projective torus in $\mathbb{P}^{s-1}$ and $r=2$ 
is treated in \cite[Theorem~18]{camps-sarabia-sarmiento-vila}.

\section{Veronese type codes}\label{Veronese-codes-section} 
Let $S=K[t_1,\ldots,t_s]$ be a polynomial
ring over a field $K$ and let $\{M_1,\ldots,M_N\}$ be the set of all
monomials of $S$ of degree $k\geq 1$, where $N=\binom{k+s-1}{s-1}$.
The map 
$$
\rho_k\colon\mathbb{P}^{s-1}\rightarrow\mathbb{P}^{N-1},\ \ \ 
[x]\mapsto [(M_1(x),\ldots,M_N(x))]
$$
is called the $k$-th {\it Veronese embedding}. Given
$\mathbb{X}\subset\mathbb{P}^{s-1}$, the $k$-th 
{\it Veronese type code\/} of degree $d$ is
$C_{\rho_k(\mathbb{X})}(d)$, the Reed--Muller-type code of degree $d$
on $\rho_k(\mathbb{X})$.
The next aim is to show that the Reed--Muller-type code
$C_{\mathbb{X}}(kd)$ has the same basic parameters and the same
generalized Hamming weights as the Veronese type code 
$C_{\rho_k(\mathbb{X})}(d)$ for $k\geq 1$ and $d\geq 1$.

\begin{lemma}\label{veronese-injective} $\rho_k$ is well-defined and injective.
\end{lemma}

\begin{proof}  If $[x]=[z]$, $x,y\in\mathbb{P}^{s-1}$, $x=(x_1,\ldots,x_s)$,
$z=(z_1,\ldots,z_s)$, then $x=\lambda z$ for some $\lambda\in
K^*$. Thus $M_i(x)=\lambda^k M_i(z)$ for all $i$, that is, 
$[(M_i(x))]=[(M_i(z))]$, here we are using $(M_i(x))$ as a short hand
for $(M_1(x),\ldots,M_N(x))$. Thus $\rho_k$ is well-defined. To show that
$\rho_k$ is injective assume that $\rho_k([x])=\rho_k([z])$. Then  for
some $\mu\in K^*$ one has $M_i(x)=\mu M_i(z)$ for all $i$. Pick $j$
such that $z_j\neq 0$ and let $\lambda=x_j/z_j$. Note that $M_i=t_j^k$
for some $i$. Then one has
$x_j^k=\mu z_j^k$, that is, $\mu=\lambda^k$. For each $1\leq \ell\leq
s$, using the monomial $M_i=t_j^{k-1}t_\ell$, 
one has
$$
x_j^{k-1}x_\ell=\mu z_j^{k-1}z_\ell=\lambda^k
z_j^{k-1}z_\ell=\lambda(\lambda z_j)^{k-1}z_\ell=\lambda(x_j^{k-1})z_\ell.
$$
Thus $x_\ell=\lambda z_\ell$ for all $\ell$, that is, $[x]=[z]$. 
\end{proof}

We come to the main result of this section. 

\begin{theorem}\label{veronese-codes} If
$\mathbb{X}\subset\mathbb{P}^{s-1}$, then the projective
Reed--Muller-type codes $C_{\rho_k(\mathbb{X})}(d)$ and
$C_{\mathbb{X}}(kd)$ have the same basic parameters and the same
generalized Hamming weights for $k\geq 1$ 
and $d\geq 1$. 
\end{theorem}

\begin{proof} Setting $N=\binom{k+s-1}{s-1}$, let 
$R=K[y_1,\ldots,y_N]=\oplus_{d=0}^\infty R_d$ be a
polynomial ring over the field $K$ with the standard grading. We can write
$\mathbb{X}=\{[P_1,],\ldots,[P_m]\}$,  
where $m=|\mathbb{X}|$, $P_i\in K^s$, and the $[P_i]$'s are in standard
form, i.e., the first non-zero entry of
$P_i$ is $1$ for
all $i$. By Lemma~\ref{veronese-injective} the map $\rho_k$ is
injective. Thus $C_{\mathbb{X}}(kd)$ and $C_{\rho_k(\mathbb{X})}(d)$
have the same length. As $[P_1],\ldots,[P_m]$ are in standard form,
for each $i$ there is $g_i\in S_{kd}$ such that $g_i(P_i)=1$. 
Therefore, by \cite[Lemma~2.13]{hilbert-min-dis}, 
we may assume that the Reed--Muller-type code $C_{\mathbb{X}}(kd)$ is the image of the 
evaluation map 
\begin{equation}\label{ev-map-reed-muller}
{\rm ev}_{kd}\colon S_{kd}=K[t_1,\ldots,t_s]_{kd}\rightarrow K^{m},\ \ \ \ \ 
g\mapsto
\left(g(P_1),\ldots,g(P_m)\right),
\end{equation}
and the Veronese type code $C_{\rho_k(\mathbb{X})}(d)$ is the image of the
evaluation map
\begin{equation}\label{ev-map-veronese}
{\rm ev}_d^1\colon R_d=K[y_1,\ldots,y_N]_d\rightarrow K^m,\ \ \ \ \ 
f\mapsto
\left(\frac{f(Q_1)}{f_1(Q_1)},\ldots,\frac{f(Q_m)}{f_m(Q_m)}\right),
\end{equation}
where $Q_i=(M_1(P_i),\ldots,M_N(P_i))$ for
$i=1,\ldots,m$, and $f_1,\ldots,f_m$ are polynomials in $R_d$ such that
$f_i(Q_i)\neq 0$ for
$i=1,\ldots,m$. For any polynomial $f=f(y_1,\ldots,y_N)=\sum \lambda_a y^a$ in $R_d$,
$\lambda_a\in K^*$, one has
\begin{eqnarray}\label{mar18-16}
f(M_1,\ldots,M_N)(P_i)&=&\sum\lambda_a (M_1^{a_1}\cdots
M_N^{a_N})(P_i)\\
&=&
\sum\lambda_a M_1^{a_1}(P_i)\cdots M_N^{a_N}(P_i)\nonumber\\
&=&f(M_1(P_i),\ldots,M_N(P_i)).\nonumber
\end{eqnarray}

As $K[t_1,\ldots,t_s]_{kd}$ is equal to $K[M_1,\ldots,M_N]_d$, any 
$g$ in $K[t_1,\ldots,t_s]_{kd}$ can be written as
$g=f(M_1,\ldots,M_N)$ for some $f=f(y_1,\ldots,y_N)$ in $R_d$.
Therefore, using Eq.~(\ref{mar18-16}), we get 
\begin{eqnarray*}
C_{\mathbb{X}}(kd)
&=&\{(g(P_1),\ldots,g(P_m))\, \vert\, g\in
K[t_1,\ldots,t_s]_{kd}\}\\
&=&\{(f(Q_1),\ldots,f(Q_m))\,
\vert\, f\in K[y_1,\ldots,y_N]_{d}\}.
\end{eqnarray*}
As a consequence, setting $\lambda_i=f_i(Q_i)$ and
$\lambda=(\lambda_1,\ldots,\lambda_m)$, one has 
\begin{equation} \label{equivalence}
C_{\mathbb{X}}(kd)=\lambda\cdot
C_{\rho_k(\mathbb{X})}(d):=\{\lambda\cdot a\, \vert\,
a\in C_{\rho_k(\mathbb{X})}(d)\},
\end{equation}
where $\lambda\cdot a:=(\lambda_1a_1,\ldots,\lambda_ma_m)$ for
$a=(a_1,\ldots,a_m)$ in $C_{\rho_k(\mathbb{X})}(d)$. This means 
that the linear codes $C_{\mathbb{X}}(kd)$ and
$C_{\rho_k(\mathbb{X})}(d)$ are equivalent
\cite[Remark~1]{dual-reed-muller}. 
Thus the dimension and minimum
distance of $C_{\mathbb{X}}(kd)$ and $C_{\rho_k(\mathbb{X})}(d)$ are
the same, and so are the generalized Hamming weights. 
\end{proof}

For convenience we recall the following classical result of S{\o}rensen
\cite{sorensen}.

\begin{theorem}{\rm(S{\o}rensen
\cite{sorensen})}\label{reed-muller-classical}
Let $K=\mathbb{F}_q$ be a finite field and let  
$C_\mathbb{X}(d)$ be the classical projective Reed--Muller code of degree $d$
on the set $\mathbb{X}=\mathbb{P}^{s-1}$. Then
$|\mathbb{X}|=(q^s-1)/(q-1)$, the minimum distance of
$C_\mathbb{X}(d)$ is given by 
$$
\delta_\mathbb{X}(d)=\left\{\hspace{-1mm}
\begin{array}{ll}\left(q-\ell+1\right)q^{s-k-2}&\mbox{ if }
d\leq (s-1)(q-1),\\
\qquad \qquad 1&\mbox{ if } d\geq
(s-1)(q-1)+1,
\end{array}
\right.
$$
where $0\leq k\leq s-2$ and $\ell$ are the unique integers such that 
$d=k(q-1)+\ell$ and $1\leq \ell \leq
q-1$, and the regularity of $S/I(\mathbb{X})$ is $(s-1)(q-1)+1$. 
\end{theorem}

Veronese codes are a natural generalization of the classical
projective Reed--Muller codes.

\begin{corollary}\cite[Proposition~1]{Renteria-Tapia-Veronese}\label{veronese-codes-corollary} If
$\mathbb{V}_k=\rho_k(\mathbb{P}^{s-1})$, then the projective
Reed--Muller-type codes $C_{\mathbb{V}_k}(d)$ and $C_{\mathbb{P}^{s-1}}(kd)$ have the
same basic parameters for $k\geq 1$ and $d\geq 1$.
\end{corollary}

\begin{proof}  This follows at once from Theorem~\ref{veronese-codes}
making $\mathbb{X}=\mathbb{P}^{s-1}$.
\end{proof}

As a byproduct we relate the dual codes
of $C_{\rho_k(\mathbb{X})}(d)$ and $C_{\mathbb{X}}(kd)$.

\begin{theorem} \label{dual}
If $\mathbb{X}$ is a subset of $\mathbb{P}^{s-1}$, then 
$C^{\perp}_{\rho_k(\mathbb{X})}(d)$ and $C^{\perp}_{\mathbb{X}}(kd)$
are equivalent codes and 
$$
C^{\perp}_{\rho_k(\mathbb{X})} (d)=\lambda \cdot C^{\perp}_{\mathbb{X}}(kd),
$$
where $\lambda=(\lambda_1,\ldots,\lambda_m)$, with
$\lambda_i=f_i(Q_i)$ for all $i=1,\ldots,m$, 
is the vector that was given in the proof of Theorem
\ref{veronese-codes}. 
\end{theorem}

\begin{proof} Let $(u_1,\ldots,u_m) \in C^{\perp}_{\mathbb{X}}(kd)$.
Then $\langle(u_1,\ldots,u_m),(v_1,\ldots,v_m)\rangle=\sum_{i=1}^m
u_iv_i=0$, for all $(v_1,\ldots,v_m) \in C_{\mathbb{X}}(kd)$. By
using Eq.~(\ref{equivalence}) we conclude that
$$
\langle(u_1,\ldots,u_m),(\lambda_1 v'_1,\ldots,\lambda_m v'_m)\rangle
=\sum_{i=1}^m u_i \lambda_i v'_i=0,
$$
for all $(v'_1,\ldots,v'_m) \in C_{\rho_k(\mathbb{X})} (d)$. Therefore
$$
\langle(\lambda_1 u_1,\ldots,\lambda_m u_m),(v'_1,\ldots,v'_m)\rangle
=\sum_{i=1}^m \lambda_i u_i v'_i=0.
$$
for all $(v'_1,\ldots,v'_m) \in C_{\rho_k(\mathbb{X})} (d)$. Thus
\begin{equation} \label{inclusion}
\lambda \cdot C^{\perp}_{\mathbb{X}}(kd) \subset C^{\perp}_{\rho_k(\mathbb{X})} (d).
\end{equation}

Furthermore one has the equalities
\begin{eqnarray} \label{dimension}
\dim_K \lambda \cdot C^{\perp}_{\mathbb{X}}(kd)&=&\dim_K
C^{\perp}_{\mathbb{X}}(kd)= m-\dim_K C_{\mathbb{X}} (kd) \nonumber \\
&= & m-\dim_K C_{\rho_k(\mathbb{X})}(d)=\dim_K C^{\perp}_{\rho_k(\mathbb{X})} (d),
\end{eqnarray}
and the equality $C^{\perp}_{\rho_k(\mathbb{X})} (d)=\lambda \cdot
C^{\perp}_{\mathbb{X}}(kd)$ follows from Eqs.~(\ref{inclusion}) and
(\ref{dimension}). Thus $C^{\perp}_{\rho_k(\mathbb{X})}(d)$ 
and  $C^{\perp}_{\mathbb{X}}(kd)$ are equivalent codes
\cite[Remark~1]{dual-reed-muller}.  
\end{proof}

\begin{corollary}
If $\mathbb{X}=\mathbb{P}^{s-1}$,
$\mathbb{V}_k=\rho_k(\mathbb{P}^{s-1})$, 
and $kd \leq (q-1)(s-1)$, then the linear code 
$C_{\mathbb{V}_k}(d)$ is equivalent to
$$
\left\{
\begin{array}{lll}
C_{\mathbb{P}^{s-1}} ((q-1)(s-1)-kd) & {\mbox{if}} & kd \not \equiv 0
\mod(q-1), \\
((1,\ldots,1),C_{\mathbb{P}^{s-1}} ((q-1)(s-1)-kd)) & {\mbox{if}} & kd \equiv 0\mod (q-1),
\end{array} \right.
$$
where $((1,\ldots,1),C_{\mathbb{P}^{s-1}} ((q-1)(s-1)-kd))$ is the
subspace of 
$K^m$ generated by $(1,\ldots,1)$ and $C_{\mathbb{P}^{s-1}} ((q-1)(s-1)-kd)$.
\end{corollary}

\begin{proof}
This result follows at once from Theorem \ref{dual} and \cite[Theorem 2]{sorensen}.
\end{proof}

The rest of this section is devoted to show some explicit examples. 

\begin{example}\label{classical-r-m-example}\rm Let $K$ be the field $\mathbb{F}_{8}$. If
$\mathbb{X}=\mathbb{P}^2$, then by Theorem~\ref{reed-muller-classical}
the basic parameters of 
the classical projective Reed--Muller-type code $C_\mathbb{X}(d)$ of degree
$d$ are given by 
\begin{small}
\begin{eqnarray*}
\hspace{-11mm}&&\left.
\begin{array}{c|c|c|c|c|c|c|c|c|c|c|c|c|c|c|c}
d & 1 & 2 & 3 & 4 & 5 & 6 & 7 & 8 & 9& 10& 11&12 &
 13& 14& 15\\
   \hline
 |\mathbb{X}| & 73 & 73 & 73 & 73 & 73 & 73 & 73 & 73 &73 &73 & 73&73 &
 73& 73&73 
 \\ 
   \hline
 H_\mathbb{X}(d)    \    & 3 & 6  & 10& 15 &21& 28& 36 & 45 &52 &58
 &63 &67 &70
 &72 &73 \\ 
   \hline
 \delta_{\mathbb{X}}(d) &64&56&48&40& 32& 24& 16& 8&7 & 6&5 & 4& 3&2
 &1 \\ 
\end{array}
\right.
\end{eqnarray*}
\end{small}
The dimension of $C_\mathbb{X}(d)$ is $H_\mathbb{X}(d)$. The
regularity of $S/I(\mathbb{X})$ is $15$ and the $a$-invariant is 
$14$. 
\end{example}

\begin{example}\rm Let $K$ be the field $\mathbb{F}_{8}$. If $k=2$, 
$\mathbb{X}=\mathbb{P}^2$, and $\mathbb{V}_2=\rho_2(\mathbb{X})$, then
by Theorem~\ref{veronese-codes} and
Example~\ref{classical-r-m-example} 
the parameters of 
the Veronese code $C_{\mathbb{V}_2}(d)$ of degree
$d$ are given by 
\begin{eqnarray*}
\hspace{-11mm}&&\left.
\begin{array}{c|c|c|c|c|c|c|c|c}
d & 1 & 2 & 3 & 4 & 5 &6 &
 7& 8\\
   \hline
 |\mathbb{V}_2|& 73 &73 &73 & 73&73 &
 73& 73&73 
 \\ 
   \hline
 H_{\mathbb{V}_2}(d)    \    & 6 & 15 & 28& 45 &58
 &67  &72
  &73 \\ 
   \hline
 \delta_{\mathbb{V}_2}(d) &56&40& 24& 8 & 6 & 4&2
 &1 \\ 
\end{array}
\right.
\end{eqnarray*}
The regularity of $S/I(\mathbb{V}_2)$ is $8$ and the $a$-invariant is
$7$.
\end{example}

\begin{example}\rm Let $K$ be the field $\mathbb{F}_{5}$. If $k=2$, 
$\mathbb{T}$ is a projective torus in $\mathbb{P}^2$, and 
$\rho_2(\mathbb{T})$ is the corresponding Veronese type code, then
by Corollary \ref{GHWtorus}, Theorem~\ref{maria-vila-hiram-eliseo},
\cite[Theorem 18]{camps-sarabia-sarmiento-vila}, 
and \textit{Macaulay}$2$ \cite{mac2}, we obtain the following information for $C_\mathbb{T}(d)$:
\begin{eqnarray*}
\hspace{-11mm}&&\left.
\begin{array}{c|c|c|c|c|c|c}
d & 1 & 2 & 3 & 4 & 5 &6\\
   \hline
 |\mathbb{T}|& 16 &16 &16 & 16&16 &
 16
 \\ 
   \hline
 H_{\mathbb{T}}(d)    \    & 3 & 6 & 10&13 &15
 &16 \\ 
   \hline
 \delta_{\mathbb{T}}(d) &12&8& 4& 3 & 2&1 \\ 
 \hline
 \delta_2(C_{\mathbb{T}}(d)) & 15 & 11 & 7 & 4 & 3 & 2 \\
\hline
 \delta_3(C_{\mathbb{T}}(d)) & 16 & 12 & 8 & 6 & 4 & 3
\end{array}
\right.
\end{eqnarray*}
and the regularity of $S/I(\mathbb{T})$ is $6$. Therefore, by 
Theorem~\ref{veronese-codes}, we get the following information  for 
the Veronese type code $C_{\rho_2(\mathbb{T})}(d)$:
\begin{eqnarray*}
\hspace{-11mm}&&\left.
\begin{array}{c|c|c|c}
d & 1 & 2 & 3\\
   \hline
 |\rho_2(\mathbb{T})|& 16 &16 &16 
 \\ 
   \hline
 H_{\rho_2(\mathbb{T})}(d)    \    & 6 & 13 & 16 \\ 
   \hline
 \delta_{\rho_2(\mathbb{T})}(d) &8&3& 1\\ 
 \hline
 \delta_2(C_{\rho_2(\mathbb{T})}(d)) & 11 & 4 & 2 \\
 \hline
 \delta_3(C_{\rho_2(\mathbb{T})}(d)) & 12 & 6 & 3 
\end{array}
\right.
\end{eqnarray*}
and the regularity of $S/I(\rho_2(\mathbb{T}))$ is $3$. 
\end{example}


\bibliographystyle{plain}

\end{document}